\documentclass[12pt]{amsart}

\usepackage{amsmath,amsthm,amsfonts,latexsym,amssymb,enumerate,color}
\usepackage{graphicx}
\usepackage{enumitem}

\newcommand\CC{{\mathbb C}}

\newcommand\DD{{\mathbb D}}

\newcommand\TT{{\mathbb T}}

\def\beq{\begin{equation}}
\def\eeq{\end{equation}}
\newtheorem{thm}{Theorem}[section]

\newtheorem{lem}[thm]{Lemma}
\newtheorem{cor}[thm]{Corollary}
\newtheorem{rem}[thm]{Remark}

\newcommand\beginpf{\noindent {\bf Proof:} \quad}

\newcommand\re{\mathop{\rm Re}\nolimits}

\def\tl{\tilde\ell}

\def\beginpf{\begin{proof}}
\def\endpf{\end{proof}}

\renewcommand\phi{\varphi}

\begin{document}
\title[Lower bounds of weighted composition operators]{The reproducing kernel thesis for lower bounds of weighted composition operators}

\author{I. Chalendar}
\address{Isabelle Chalendar, Universit\'e Paris Est, 
LAMA, (UMR 8050), UPEM, UPEC, CNRS,
F-77454, Marne-la-Vall\'ee (France)}
\email{isabelle.chalendar@u-pem.fr}

\author{J.R.Partington}
\address{Jonathan R. Partington, School of Mathematics, University of Leeds, Leeds LS2 9JT, UK}
 \email{J.R.Partington@leeds.ac.uk}
 
 %\noindent\textsc{Mathematics Subject Classification} (2000):
 %Primary:  47A15, 47D03
 %Secondary:  30H10, 31C25
 
\baselineskip18pt

\bibliographystyle{plain}

\begin{abstract}
It is shown that the property of being bounded below (having closed range) of
weighted composition operators on Hardy and Bergman spaces can be tested
by their action on a set of simple test functions, including reproducing kernels.
The methods used in the analysis are based on the theory of reverse Carleson embeddings.
\end{abstract}

 \maketitle

\section{Introduction and notation}

The reproducing kernel thesis is a term commonly used to describe a body of
results that assert that the boundedness of various operators on function spaces
such as the Hardy and Bergman space can be tested by their action on
reproducing kernels: this is known to apply to Hankel operators, Toeplitz
operators, Carleson embeddings, and -- our main concern here -- weighted composition
operators, although not to the adjoints of weighted composition operators (see \cite{zen}, also \cite{CZ07,GKP10}).

Surprisingly, the reproducing kernels may also be used in some circumstances as
test functions for boundedness below of certain operators, and this is the theme of
this paper. We give necessary and sufficient conditions for a bounded weighted
composition operator $W$ on a Hardy space $H^p$ or a Bergman space $A^p$ to
be bounded below, that is, for there to exist a constant $\delta >0$ such that
$\|Wf\| \ge \delta\|f\|$ for all $f$ in the space. Assuming that the operator is injective,
which it is except in trivial cases, this property is equivalent to the property that $W$
has closed range. Some questions remain open in the case of Bergman spaces,
particularly for the case $p=2$ if the weight has infinitely many zeros.

The results of this paper generalize in various directions
earlier work from \cite{CTW,gp17,gt15,kumar,luery}, as will be seen below.

We write $m$ for normalized Lebesgue measure on $\TT$ and $A$ for normalized
area measure in $\DD$.

\section{Hardy spaces}

For $h \in H^p=H^p(\DD)$ and $\psi: \DD \to \DD$ holomorphic 
we define the weighted composition operator $W_{h,\psi}$ on $H^p$
by
\[
W_{h,\psi}f(z)=h(z) f(\psi(z)), \qquad (z \in \DD).
\]
If $h \in H^\infty$, then $W_{h,\psi}$ is automatically bounded,
by Littlewood's subordination theorem, but this is not a necessary condition
for boundedness.

For $\lambda\in \DD$, we write $k_\lambda$ for the reproducing kernel function
$ z \mapsto 1/(1-\overline\lambda z)$ for $z \in \DD$.
Using the (isomorphic) duality between $H^p$ and
$H^{p'}$, where $p'=p/(p-1)$ is the conjugate index to $p$
\cite[A5.7.8]{nik}, we see that there are constants $A,B>0$ independent of $\lambda$ such that
for $1<p<\infty$ we have 
\[
A \|\delta_\lambda\|_{(H^{p'})^*}  \le \|k_\lambda\|_p \le B \|\delta_\lambda\|_{(H^{p'})^*},
\]
where $\delta_\lambda$ is the functional $f \mapsto f(\lambda)$.
Now,   the standard inner--outer factorization shows that every $f \in H^{p'}$
can be written as $f = \theta u^{2/p'}$ for $\theta $ inner and  $u \in H^2$ outer, and conversely every
$g \in H^2$ can be written as $g=\phi v^{p'/2}$ for $\phi$ inner and $v \in H^{p'}$ outer.
We conclude easily that
\[
\| \delta_\lambda\|_{(H^{p'})^*} = \|\delta_\lambda\|_{H^2}^{2/p'} =(1-|\lambda|^2)^{-1/p'}.
\]

We write $\tilde k_\lambda=k_\lambda / \|k_\lambda\|_p$,
noting that this definition depends on $p$.

As in \cite{CH01} we define the measure $\mu_{h,\psi}$ on Borel
subsets of $\overline\DD$ by
\[
\mu_{h,\psi}(E)=\int_{\psi^{-1}(E) \cap \TT} |h|^p dm,
\]
where $m$ is normalized Lebesgue measure. In \cite{CH01} it is shown that
$W_{h,\psi}$ is bounded on $H^p$ if and only if $\mu_{h,\psi}$ is
a Carleson measure: this follows from the fact that
\beq\label{eq:normwf}
\|W_{h,\psi}f\|^p = \int_{\overline\DD} |f|^p \, d\mu_{h,\psi}
\eeq
for $f \in H^p$.

Cima, Thomson and Wogen \cite{CTW} showed that the composition
operator $C_\phi$ on $H^2$ (the case $h=1$, $p=2$) has closed range
if and only if the Radon--Nikodym derivative
\beq\label{eq:rnderiv}
g_{h,\psi}=\frac{d\mu_{h,\psi} }{dm}
\eeq
is essentially bounded away from $0$ on $\TT$. See also \cite{Z}
for another characterization. This was extended to $H^p$ by
Galanopoulos  and  Panteris \cite{gp17}.

In \cite{kumar}, the result was extended to weighted composition
operators on $H^2$. In fact a similar argument gives the full
result for $H^p$.

\begin{lem}\label{lem:aftercima}
Let $1<p<\infty$ and let $h \in H^p$ and $\psi: \DD \to \DD$ holomorphic 
be such that $W_{h,\psi}$ is bounded. Then $W_{h,\psi}$ is bounded
below if and only if $g_{h,\psi}$, as defined in \eqref{eq:rnderiv},
is essentially bounded away from $0$ on $\TT$.
\end{lem}
\beginpf
It follows from \eqref{eq:normwf} that 
\[
\|W_{h,\psi}f\|^p \ge \int_\TT |f|^p g_{h,\psi} dm,
\]
and so $W_{h,\psi}$ is bounded below if $g_{h,\psi}$
is essentially bounded away from $0$. 

Conversely, if  $g_{h,\psi}$
is not essentially bounded away from $0$, then for each $\epsilon>0$ there is
a set $E \subset \TT$ such that $m(E)>0$ and
\[
\int_{\psi^{-1}(E) \cap \TT} |h|^p \, dm < \epsilon m(E).
\]
As in \cite[p. 24]{durenbook}, for example, there exists a function $f\in H^p$ such that
\[
|f(e^{i\theta})| =
\begin{cases}
1 & \hbox{if } e^{i\theta} \in E, \\
\frac12 & \hbox{if } e^{i\theta} \in \TT \setminus E.
\end{cases}
\]
Now for $n=1,2,\ldots$ we have $\|f^n\|_p \ge m(E)^{1/p}$ but 
$f^n \to 0$ pointwise on $\TT \setminus E$, so
\[
\limsup_{n \to \infty} \|W_{h,\psi}f^n\|_p^p = \int_{\psi^{-1}(E) \cap \TT} |h(e^{i\theta})|^p \, dm \le \epsilon m(E).
\]

\endpf

This leads to a reproducing kernel thesis for boundedness below on $H^p$.

\begin{thm}\label{thm:lplus}
Let $1<p<\infty$ and let $h \in H^p$ and $\psi: \DD \to \DD$ be holomorphic 
with $W_{h,\psi}$  bounded. The following assertions are
equivalent:\\
(i) $W_{h,\psi}$ is bounded below;\\
(ii) There exists $C>0$ such that 
$ \|W_{h,\psi}   k_\lambda \|_p \ge C \|k_\lambda\|_p$
for all $\lambda \in \DD$.
\end{thm}

\beginpf
By Theorem 2.1 in \cite{HMNO}
(see also  \cite{FHR}), the function $g_{h,\psi}$ is essentially 
bounded away from $0$ if and only if
there is a constant $C>0$ such that 
\[
\int_{\overline \DD} |\tilde k_\lambda(z)|^p \, d\mu_{h,\psi}
\ge
C \qquad (\lambda \in \DD).
\]
%In fact the result in \cite{HMNO} is given only for $p=2$, but
%the arguments are valid for $1<p<\infty$ with obvious modifications.

Using \eqref{eq:normwf} and Lemma \ref{lem:aftercima} we have the result.
\endpf

For unweighted composition operators, and $p=2$, this
result may be found in the thesis of Luery \cite{luery}.

Now for $1 \le p < \infty$, let
$\tl_w \in H^p$ be defined by
\[
\tl_w (z)=\frac{(1-|w|^2)^{1/p}}{(1-\overline w z)^{2/p}},
\]
so that $\| \tl_w \|_{H^p} =1$ for all $w \in \DD$. We may use these
test functions for boundedness below of weighted composition
operators on $H^p$.

\begin{thm}\label{thm:othertests}
Let $1\leq p<\infty$,  $h \in H^p$, and $\psi: \DD \to \DD$ be
holomorphic   such that the weighted composition operator $W_{h,\psi}$ is
bounded. Then $W_{h,\psi}$ is bounded below if and only if there is a constant
$C>0$ such that $\|W_{h,\psi} \tl_w \| \ge C$ for all $w \in \DD$.
\end{thm}
\beginpf
Without loss of generality, we may assume that $h$ is outer, and hence non-vanishing,
since if $h=\theta u$ with $\theta$ inner and $u$ outer, then the operator 
$W_{u,\psi}$ is bounded below if and only if $W_{h,\psi}$ is; a similar
observation applies to boundedness below on test functions.
The   condition on test functions may be written as 
\[
\int_\TT |h(z)|^p |\tl_w(\psi(z))|^p \, dm(z) \ge C^p,
\]
and since $h$ is non-vanishing we may write $h(z)^p = \tilde h(z)^2$, where $\tilde h \in H^2$.
Thus with $\tilde k_w=k_w/\|k_w\|_2$ we have
\[
\int_\TT |\tilde h(z)|^2 |\tilde k_w(\psi(z))|^2 \, dm(z) \ge C^p
\]
for all $w \in \DD$, and so the weighted composition
operator $W_{\tilde h, \psi}$ is bounded below on $H^2$, by Theorem  \ref{thm:lplus}.
It now follows from Lemma \ref{lem:aftercima} that $W_{h,\psi}$ is bounded below on $H^p$.
\endpf

As a corollary of Theorem \ref{thm:lplus} we have a reproducing kernel thesis for
boundedness below of composition operators $C_\Phi$
on the right half-plane $\CC_+$. Note that these operators
are not automatically  bounded, but an exact expression for
their norm is given in \cite{EJ12}.

By means of a unitary equivalence between $H^2(\DD)$ and $H^2(\CC_+)$
induced by the self-inverse M\"obius bijection $M(z)=(1-z)/(1+z)$, namely,
\[
(Vf)(s)=\frac{1}{\sqrt\pi(1+s)} f\left( \frac{1-s}{1+s} \right), \qquad f \in H^2(\DD), \quad s \in \CC_+,
\]
the composition operator $C_\Phi$ on $H^2(\CC_+)$ is seen to
be unitarily equivalent to the weighted composition operator $W_{h,\phi}$
on  $H^2(\DD)$, where $h(z)=\frac{1+\phi(z)}{1+z}$ and $\phi= M \circ \Phi \circ M$ (see \cite{CP03,kumar}).

Let $\tilde K_w$, given by 
\[
\tilde K_w(s)=\frac{(2\pi \re w)^{1/2}}{s+\overline w} ,
\]
denote the normalized reproducing kernel at $w \in \CC_+$. Since
$\langle F, Vk_\lambda \rangle = ( V^{-1}F)(\lambda)$ for $F \in H^2(\CC_+)$,
we conclude that $V \tilde k_\lambda = \tilde K_w$, where
$ w=M(\lambda)$, and hence obtain the following corollary, which can also be proved fairly
directly.

\begin{cor}\label{cor:h2cp}
If the composition operator $C_\Phi$ is bounded on $H^2(\CC_+)$,
then it is bounded below if and only if there is a constant $C>0$ such that
$\|C_\Phi \widetilde K_w \| \ge C$ for all $w \in \CC_+$.
\end{cor}
Clearly similar results hold for weighted composition operators, and
also for $H^p(\CC_+)$ for other values of $p$.

% : however, in the latter
% case the test functions are no longer reproducing kernels.

Another easy corollary of the main theorem of \cite{HMNO} is
a reproducing kernel thesis for Toeplitz operators.
If $J: H^2 \to L^2(\TT, \mu)$ is a bounded (Carleson) embedding,
then it is clearly bounded below if and only if $J^*J$ is bounded
below as an operator on $H^2$ (consider $\langle J^*Jx,x\rangle$).
 
\begin{cor}
Suppose that $h \in L^\infty(\TT)$ with $h \ge 0$ a.e. Then
the Toeplitz operator $T_h:f \mapsto P_{H^2}( h.f)$ is bounded
below if and only if
it is bounded  
below on normalized reproducing kernels.
\end{cor}
 \beginpf
 Let $\mu$ denote the measure with Radon--Nikodym derivative $h$, so
 that
 \[
\langle  J^*J f , g \rangle = \int_{\TT} f \overline g h \, dm = 
\langle T_h  f, g \rangle
\]
for $f,g \in H^2$. Since we can test $J^*J$ on normalized reproducing kernels,
by \cite{luery}, or indeed Theorem \ref{thm:lplus},
the result follows.
\endpf

\section{Bergman spaces}

We now consider weighted composition operators $W_{h,\psi}$
acting on the Bergman space $A^p=A^p(\DD)$.
Once again a measure is associated with such an operator,
this time $\mu_{h,\psi}^p$ defined on Borel subsets of the
disc by
\beq\label{eq:revap1}
\mu_{h,\psi}^p(E)= \int_{\psi^{-1}(E)} |h(z)|^p \, dA(z),
\eeq
and   we have
\beq\label{eq:revap2}
\|W_{h,\psi} f\|^p = \int_\DD |f(z)|^p \, d\mu_{h,\psi}^p(z).
\eeq
This is done in \cite[Lem.\ 3.1]{kumar} for the case $p=2$, but the
argument works for all $p$.
It follows that $W_{h,\psi}$ is bounded   and bounded below
if and only if  $\mu^p_{h,\psi} $ satisfies the Carleson and
reverse Carleson properties. The unweighted case of this
result for $p=2$ may be found in \cite{Z}.

We begin with the case $p=2$ and write $\mu_{h,\psi}$ for $\mu_{h,\psi}^2$
and $g_{h,\psi}$ for the
Radon--Nikodym derivative of  $\mu_{h,\psi}$. We have, using Corollary 1 of
\cite{luecking}, that $W_{h,\psi}$ is bounded  below 
if and only if there exist constants $\delta, C>0$ such that
\[
A (D \cap \{z \in \DD: g_{h,\psi} (z) > \delta \})
\ge C A(D \cap \DD)
\]
for all discs $D$ with centres on $\TT$. (See \cite[Thm. 3.1]{kumar}.)

More recently, Ghatage and  Tjani \cite{gt15} have analysed the unweighted
case by means of the Berezin transform: in our context we define
it by
\beq\label{eq:berezin}
\tilde \mu(h,\psi)(w)= \int_\DD |\tilde k_w (z)|^2 d\mu_{h,\psi}(w)
= \|W_{h,\psi}\tilde k_w\|^2,
\eeq
where now $\tilde k_w$ is the normalized Bergman kernel,
\[
\tilde k_w(z)= \frac{1-|w|^2}{(1-\overline w z)^2}, \qquad w,z \in \DD.
\]
The following theorem gives an extension of \cite{gt15} to weighted
composition operators.

\begin{thm}\label{thm:a2}
For a bounded weighted composition operator $W_{h,\psi}$ on
$A^2(\DD)$ the following conditions are equivalent:\\
(i) $W_{h,\psi}$ is bounded below;\\
(ii) $\mu_{h,\psi}$ satisfies the reverse Carleson condition;\\
(iii) $\tilde \mu(h,\psi)(w)$ is bounded away from zero; that is,
there is a constant $C>0$ such that
$\|W_{h,\psi} \tilde k_w \| \ge C$ for all $w$.
\end{thm}

\begin{proof}
The equivalence of (i) and (ii) is given in \cite[Lem.\ 3.1]{kumar};
the equivalence of (ii) and (iii) is contained in Theorem 4.1 of \cite{gt15},
which asserts that a measure satisfies the reverse Carleson condition if
and only if its Berezin transform is bounded away from $0$, together with
\eqref{eq:berezin}.
\end{proof}

\begin{rem}{\rm
In the case of weighted composition operators on weighted Bergman spaces $A^2_\alpha$, with 
$\alpha>-1$ and the
norm   given by
\[
\|f\|^2_{A^2_\alpha} = \int_\DD |f(z)|^2 (1-|z|^2)^\alpha \, dA(z),
\]
we still have the equivalence of (i) and (ii) in Theorem~\ref{thm:a2}, since the proof of Lemma 3.1 in
\cite{kumar} is easily seen to extend to this situation. However, at present we do not know
whether the equivalence with (iii) still holds.
}
\end{rem}

As with Corollary \ref{cor:h2cp} we may
obtain a corollary for  composition
operators on the Bergman space of $\CC_+$. We note that the norm of a bounded composition
operator on $A^2(\CC_+)$ is given in \cite{ew11}. In the following result $\widetilde K_w$ is the
normalized reproducing kernel for $A^2(\CC_+)$.

\begin{cor}
If the composition operator $C_\Phi$ is bounded on $A^2(\CC_+)$
then it is bounded below if and only if there is a constant $C>0$ such that
$\|C_\Phi \widetilde K_w \| \ge C$ for all $w \in \CC_+$.
\end{cor}

Since the proof is very similar to the proof of Corollary \ref{cor:h2cp}, we omit it.\\

Now for $1 \le p < \infty$, let
$\tl_w \in A^p$ be defined by
\[
\tl_w (z)=\frac{(1-|w|^2)^{2/p}}{(1-\overline w z)^{4/p}},
\]
so that $\| \tl_w \|_{A^p} =1$ for all $w \in \DD$. We may use these
test functions for boundedness below of weighted composition
operators on $A^p$. The following theorem corresponds to
Theorem \ref{thm:othertests} for the Hardy space, but requires a supplementary condition
on $h$, as we do not have a suitable inner--outer factorization available.

\begin{thm}\label{thm:finmanyzeros}
Let $1\leq p<\infty$,  $h \in A^p$ with at most finitely-many zeros, and $\psi: \DD \to \DD$
holomorphic such that the weighted composition operator $W_{h,\psi}$ is
bounded. Then $W_{h,\psi}$ is bounded below if and only if there is a constant
$C>0$ such that $\|W_{h,\psi} \tl_w \| \ge C$ for all $w \in \DD$.
\end{thm}
\beginpf
As in the proof of Theorem \ref{thm:othertests} we may assume without loss of generality
that $h$ has no zeros. This time we divide out its zeros by a contractive divisor $G$,
as in \cite{DKSS93,hedenmalm}. Since $G$ is analytic on a neighbourhood of the
disc, it is also bounded, and thus plays the same role as the inner function $\theta$ did
in the Hardy space. That is, $W_{h,\psi}$ and $W_{h/G,\psi}$ are both bounded below (or not)
together.

The  condition on test functions may be written as
\[
\int_\DD |h(z)|^p |\tl_w(\psi(z))|^p \, dA(z) \ge C^p,
\]
and since $h$ is non-vanishing we may write $h(z)^p = \tilde h(z)^2$, where $\tilde h \in A^2$.
Thus
\[
\int_\DD |\tilde h(z)|^2 |\tilde k_w(\psi(z))|^2 \, dA(z) \ge C^p
\]
for all $w \in \DD$, and so the weighted composition
operator $W_{\tilde h, \psi}$ is bounded below on $A^2$, by Theorem  \ref{thm:a2}.

Looking  at \eqref{eq:revap1} and \eqref{eq:revap2},  noting that
$\mu_{h,\psi}^p = \mu_{\tilde h,\psi}^2$, and observing that Luecking's
condition for a reverse Carleson measure \cite[Cor.~1]{luecking} is independent of $p$, we see
that
 $W_{h,\psi}$ is bounded below on $A^p$.
\endpf

It would be interesting to know whether Theorem \ref{thm:finmanyzeros} extends to the
case when $h$ has infinitely-many zeros, and the corresponding contractive divisor $G$
may not be bounded.

% \bibliographystyle{plain}
% \bibliography{biblio-lowerbd}

\end{document}